\documentclass[a4paper,UKenglish,cleveref, autoref, thm-restate]{lipics-v2021}

\usepackage[noend]{algorithm2e}
\usepackage{csquotes}
\usepackage{amsthm}
\usepackage{amsmath}
\usepackage{amsfonts}
\usepackage{faktor}
\usepackage{bm}
\usepackage{booktabs}
\usepackage[normalem]{ulem}
\usepackage{graphicx}
\usepackage{subcaption}
\usepackage{mathtools}
\usepackage{float}
\usepackage{capt-of} 
\usepackage[small, bold, full]{complexity}
\useunder{\uline}{\ul}{}

\usepackage{enumerate}

\usepackage[table]{xcolor}
\definecolor{gryffindor}{RGB}{220,0,1}
\definecolor{slytherin}{RGB}{26,121,42}
\definecolor{hufflepuff}{RGB}{236,185,57}
\definecolor{ravenclaw}{RGB}{14,26,164}

\usepackage[colorinlistoftodos,prependcaption,obeyFinal]{todonotes}

\title{Twisted Rational Zeros and Local-Global Principles for Linear Recurrence Sequences}

\author{Piotr {Bacik}}{University of Oxford, UK \and Max Planck Institute for Software Systems, Saarland Informatics Campus, Germany}{piotr.bacik@stcatz.ox.ac.uk}{https://orcid.org/0009-0006-0248-3204}{Supported by EPSRC grant EP/X033813/1.}

\authorrunning{P. Bacik}

\Copyright{Piotr Bacik}

\acknowledgements{I am grateful to Florian Luca (who sadly passed away during the writing of this paper) for suggesting a key idea to reduce the proofs of the main theorems in this paper to considering reductions of algebraic numbers modulo a prime. I am also grateful to Joris Nieuwveld for providing the idea for the heuristic argument in \cref{sec:conclusion}. I also thank Jo\"el Ouaknine and James Worrell for helpful comments on the layout of the paper.}

\nolinenumbers



\category{}

\newtheorem{problem}[definition]{Problem}

\newcommand{\QQ}{\mathbb{Q}}

\newcommand{\pp}{\mathfrak{p}}
\renewcommand{\PP}{\mathfrak{P}}
\newcommand{\qq}{\mathfrak{q}}
\newcommand{\ZZ}{\mathbb{Z}}
\newcommand{\NN}{\mathbb{N}}
\newcommand{\KK}{\mathbb{K}}
\newcommand{\LL}{\mathbb{L}}

\renewcommand{\O}{\mathcal{O}}

\newcommand{\Alg}{\smash{\overline{\QQ}}}

\newclass{\EqSLP}{EqSLP}
\newcommand{\ord}{\mathrm{ord}}

\numberwithin{equation}{section}

\begin{document}

\maketitle

\begin{abstract}
The Skolem Problem asks whether a given linear recurrence sequence (LRS) has a zero term, and is equivalent to proving an effective version of the Skolem-Mahler-Lech theorem, which states that a non-degenerate LRS has finitely many zeros. Decidability of the Skolem Problem however, has remained open for many decades. 

Bilu \emph{et al.\@} (2022) showed that the Skolem Problem for simple LRS is decidable subject to the weak $p$-adic Schanuel Conjecture and the Skolem Conjecture (also known as the exponential local-global principle), the latter of which states that an LRS has an integer zero if and only if it has a zero modulo every integer $m$. This paper works towards understanding the Skolem Conjecture. We provide an example showing that a natural strengthening of the Skolem Conjecture, where we restrict $m$ to be a prime power, is false. This failure is accounted for by the relationship between $p$-adic zeros of LRS (which arise naturally in the proof of the Skolem--Mahler--Lech theorem and have been studied algorithmically by Bacik \emph{et al.\@} (2026)) and twisted rational zeros (introduced by Bilu \emph{et al.\@} (2025)). 

By studying this relationship further, we are able to prove the main result of this paper: a local-global principle for simultaneous zeros of two coprime LRS, subject to the $p$-adic Schanuel Conjecture. A fundamental step in proving this result is characterising when twisted rational zeros are (or are not) $p$-adic zeros for infinitely many primes $p$, which we do unconditionally. This also answers two open questions of Bilu \emph{et al.\@} (2025). Finally, we conjecture that the existence of twisted rational zeros is the only way that the aforementioned strengthened Skolem Conjecture may fail, which is supported by a heuristic argument.
\end{abstract}
\newpage

\section{Introduction}
\subsection{The Skolem Problem}
A sequence $\langle u_n \rangle_{n=0}^\infty$ of rational numbers is called a linear recurrence sequence (LRS) if it satisfies a linear recurrence relation
\begin{align} \label{eqn:LRS-rec}
    a_d u_{n+d} = a_{d-1} u_{n+d-1} + \dots + a_0 u_n \, ,
\end{align}
where $a_0, \dots, a_{d} \in \ZZ$, and $a_0,a_d \neq 0$. We say $u$ has order $d$ if $u$ satisfies \eqref{eqn:LRS-rec} and does not satisfy any recurrence of smaller length. Since $a_0 \neq 0$, the sequence extends uniquely to a bi-sequence $\langle u_n \rangle_{n=-\infty}^\infty$, and we use this extension whenever negative indices occur. An \emph{integer zero} of an LRS $u$ is an integer $n \in \ZZ$ for which $u_n = 0$. The celebrated Skolem--Mahler--Lech theorem \cite{Skolem_SML,Mahler_SML,lech_note_1953} states that the set of integer zeros $\{n \in \ZZ: u_n = 0\}$ of any LRS $u$ is a union of finitely many arithmetic progressions and a finite set.

The Skolem--Mahler--Lech theorem may be refined by the concept of non-degeneracy. Recall that every LRS has an exponential-polynomial form
\begin{align} \label{eqn:LRS-exp-poly-intro}
    u_n = \sum_{i=1}^s P_i(n) \lambda_i^n
\end{align}
where $P_1,\dots,P_s$ are polynomials with algebraic coefficients and $\lambda_1, \dots, \lambda_s$ are distinct algebraic numbers. We say $u$ is degenerate if some ratio $\lambda_i/\lambda_j$ is a root of unity for $i \neq j$, and non-degenerate otherwise. Note that any LRS may be decomposed into non-degenerate subsequences $\langle u_{Mn+\ell} \rangle_{n=0}^\infty$ for each $\ell = 0 ,\dots, M-1$, by taking $M$ to be the lowest common multiple of the orders of any ratios $\lambda_i/\lambda_j$ that are roots of unity. The core of the Skolem--Mahler--Lech theorem is that a non-degenerate LRS has finitely many integer zeros.

Unfortunately, the proof of this result is ineffective; there is no known algorithm to compute the set of zeros, or decide whether it is empty. This is exactly the well-known Skolem Problem:
\begin{problem}[Skolem Problem]
Given an LRS $u$, decide whether there exists $n \in \NN$ such that $u_n = 0$.
\end{problem}
The Skolem Problem is a central question in the study of reachability problems, appearing in areas as diverse as loop termination \cite{ouaknine_termination_2015,Karimov_2022}, control theory \cite{blondel_2000}, and theoretical biology~\cite{Soittola_1976}. However, to date, decidability of the Skolem Problem remains stubbornly open, to the point where the Skolem Problem is regularly used as a benchmark of hardness; two examples being in reachability for Markov Chains \cite{vahanwala_2024}, and in computing invariants of probabilistic polynomial programs \cite{mullner_strong_2024}.

Decidability is known, however, in some special cases; for instance, it is known for sequences of order at most 4 \cite{tijdeman_mst1984,vereshchagin_1985,bacik_completing_2025}, but open for sequences of order 5 and above. If one assumes some number-theoretic conjectures, then one can go further. For example, decidability is known for all LRS subject to a strengthening of the Cram\'er--Granville conjecture \cite{luca_2025}, though the algorithm is impractical and the conjecture is required for both termination \emph{and} correctness. 

\subsection{The exponential local-global principle} This paper is motivated by questions surrounding another conjectural result. We say that $u$ is \emph{simple} if it satisfies \eqref{eqn:LRS-exp-poly-intro} and $\deg P_i = 0$ for all $i = 1 ,\dots, s$. The following was conjectured by Skolem \cite{skolem_1937},\footnote{Skolem's original statement was imprecise, but was made more precise by subsequent authors \cite{schinzel_1977,bartolome_exponential_2013}.} and is also known as the exponential local-global principle.
\begin{conjecture}[Skolem Conjecture] \label{conj:skolem-conj}
Suppose that $u$ is a simple LRS such that for some integer $b \geq 1$, $u_n \in \ZZ[1/b]$ for all $n \in \ZZ$. Then $u$ has no integer zero if and only if there exists an integer $m \geq 1$ with $\gcd(m,b) = 1$ such that $u_n \not\equiv 0 \bmod m$ for all $n \in \ZZ$.  
\end{conjecture}
Bilu \emph{et al.} \cite{bilu_skolem_2022} show that this conjecture implies decidability of the Skolem Problem for simple LRS, if one also assumes the $p$-adic Schanuel Conjecture. We briefly explain the role of these conjectures: the Bi-Skolem Problem asks, given LRS $u$, whether there exists $n \in \ZZ$ with $u_n = 0$. Decidability of the Bi-Skolem Problem for simple LRS easily follows from the Skolem Conjecture. Indeed, the algorithm proceeds by searching for $n \in \ZZ$ with $u_n = 0$, and in parallel, searching for integer $m \geq 1$ with $u_n \not\equiv 0 \bmod m$ for all $n \in \ZZ$. Since any LRS taking values in $\ZZ[1/b]$ is ultimately periodic modulo any integer $m$ coprime to $b$, one can easily check whether $u_n \not\equiv 0 \bmod m$ for all $n \in \ZZ$ by computing $u_n$ for sufficiently many values of $n$. \cref{conj:skolem-conj} implies that this algorithm must terminate. On the other hand, the $p$-adic Schanuel Conjecture is used to prove that the Skolem Problem reduces to the Bi-Skolem Problem. Notably, the algorithm this gives for the Skolem Problem for simple LRS relies on the conjectures only for termination, not correctness.

\subsection{\texorpdfstring{$p$}{p}-adic zeros, and a strengthening of Skolem's Conjecture}

The notion of a $p$-adic zero of an LRS arises naturally, due to the role of $p$-adic numbers in the proof of the Skolem--Mahler--Lech theorem. We give a brief definition of $p$-adic numbers and $p$-adic zeros. For any prime $p$, and any non-zero rational number $x \in \QQ$, write $x = \frac{a}{b}p^r$ where $a,b,p$ are pairwise coprime. Then we define the $p$-adic valuation by $v_p(x) = r$, and $v_p(0) = \infty$, and we define the $p$-adic absolute value by $|x|_p = p^{-v_p(x)}$. 

The set of $p$-adic numbers $\QQ_p$ is defined by taking the Cauchy completion of $\QQ$ under $|\cdot|_p$, and the set of $p$-adic integers $\ZZ_p$ is the Cauchy completion of $\ZZ$ under $|\cdot|_p$. Alternatively, $\ZZ_p$ may be characterised as all $p$-adic numbers $x \in \QQ_p$ with $v_p(x) \geq 0$. Every $p$-adic integer has a hands-on representation as an infinite sum $\sum_{k=0}^\infty a_k p^k$ which is convergent with respect to $|\cdot|_p$, and where $a_k \in \{0,1,\dots,p-1\}$. For example, we have $-1/2 \in \ZZ_3$, represented by the $3$-adically convergent sum $-1/2 = \sum_{k=0}^\infty 3^k$.

We call $x \in \ZZ_p$ a $p$-adic zero of an LRS $u$, if there is a sequence of integers $n_j \in \ZZ$ such that $n_j \to x$ and $u_{n_j} \to 0$ in $p$-adic absolute value as $j \to \infty$. Equivalently, $v_p(n_j - x) \to \infty$ and $v_p(u_{n_j}) \to \infty$ as $j \to \infty$. An algorithm for computing $p$-adic zeros of an LRS was given in \cite{bacik_p-adic_conference_2026}, with termination subject to the $p$-adic Schanuel Conjecture.

Now, every integer zero of an LRS is a $p$-adic zero for every prime $p$. Suppose $u$ is an LRS satisfying \eqref{eqn:LRS-rec}, then say a prime $p$ is regular if $p\nmid a_0,a_d$. The proof of the Skolem--Mahler--Lech theorem proceeds by proving that a non-degenerate LRS $u$ satisfying \eqref{eqn:LRS-rec} has finitely many $p$-adic zeros for any regular prime $p$, therefore implying that a non-degenerate sequence has finitely many integer zeros. Motivated by the Hasse-Minkowski theorem \cite[Section 3.2]{serre_course_1973}, which is a local-global principle for quadratic forms, one might conjecture the following.
\begin{conjecture} \label{conj:failed-conj}
If $u$ is a simple LRS which has a $p$-adic zero for every regular prime $p$, then there exists $n \in \ZZ$ with $u_n = 0$.
\end{conjecture}
This corresponds to restricting $m$ to be a prime power in \cref{conj:skolem-conj}. It turns out, however, that this conjecture fails.
\begin{example} \label{ex:intro-ex}
Let $u_n = (1+2^n)(1+(-4)^n) = 1 + 2^n + (-4)^n + (-8)^n$, which satisfies
\begin{align*}
    u_{n+4} = -9 u_{n+3} + 2 u_{n+2} + 72 u_{n+1} -64 u_n \, .
\end{align*}
It is clear that $1+2^n \neq 0$ and $1+(-4)^n \neq 0$ for all $n \in \ZZ$, and so $u_n \neq 0$ for all $n \in \ZZ$. However, 0 is a $p$-adic zero for every prime $p>2$. For a proof, see \cref{ex:bad-TRZ-example}.
\end{example}
This example shows the perhaps surprising fact that \textbf{integers that are $p$-adic zeros of LRS need not be genuine integer zeros.} Taking a particular prime for a concrete example, if we consider $w_n = 1 + 2^n$, then 0 is a $3$-adic zero, witnessed by the sequence of integers $n_j = 3^j$. Indeed, we have that $v_3(n_j) \to \infty$ and $v_3(w_{n_j}) \to \infty$ as $j \to \infty$,\footnote{The fact that $v_3(w_{3^j}) \to \infty$ as $j \to \infty$ follows from the proofs in \cref{ex:bad-TRZ-example}, but it may also be proven in an elementary way by induction that $v_3(w_{3^j}) \geq j+1$. Indeed, this holds for $j=0$, establishing the base case, while in the general case, we have $w_{3^{j+1}} = (2^{3^j}+1)(2^{2\cdot 3^j}- 2^{3^j} + 1)$. By the induction hypothesis, $v_3(2^{3^j}+1) \geq j+1$, and for the second factor we have $2^{2\cdot 3^j}- 2^{3^j} + 1 \equiv (-1)^2 - (-1) + 1 \equiv 0 \bmod 3$, so $v_3(w_{3^{j+1}}) \geq j+2$, proving the claim.} identifying $0$ as a $3$-adic zero.

On the other hand, our main result of this paper is to show that by understanding this phenomenon better, we can actually prove (conditionally) a local-global principle for simultaneous zeros of LRS. Suppose that $u$ satisfies \eqref{eqn:LRS-exp-poly-intro}. Then if there exist integers $m_1,\dots,m_s$ and a root of unity $\xi \neq 1$ such that $\lambda_1^{m_1} \cdots \lambda_s^{m_s} = \xi$, we say $u$ is \emph{weakly degenerate}. Any LRS may be decomposed into subsequences which are each not weakly degenerate, and if $u$ is not weakly degenerate, we may write $u_n = P(n,\mu_1^n, \dots, \mu_r^n)$ where $\mu_1, \dots, \mu_r$ are algebraic numbers and $P$ is a Laurent polynomial (see \cref{sec:prelims:LRS}). If, in addition, we may write $v_n = Q(n, \mu_1^n, \dots, \mu_r^n)$, and $P,Q$ are coprime as Laurent polynomials, then we say $u,v$ are coprime. 
\begin{theorem}[Simultaneous Exponential Local-Global Principle for Coprime LRS] \label{thm:sim-local-global-intro}
Suppose that $u,v$ are coprime LRS that are not weakly degenerate, taking values in $\ZZ[1/b]$, for some integer $b \geq 1$. Then, subject to the $p$-adic Schanuel conjecture, there does not exist $n \in \ZZ$ with $u_n = v_n = 0$ if and only if there exists an integer $m \geq 1$ with $\gcd(m,b) =1 $ such that $(u_n,v_n) \not\equiv (0,0) \bmod m$ for all $n \in \ZZ$.

Moreover, we may take $m = (pq)^k$, for some integer $k \geq 1$ and primes $p,q$.
\end{theorem}
This immediately yields an algorithm (with termination subject to the $p$-adic Schanuel Conjecture) for deciding whether there exists $n \in \ZZ$ with $u_n = v_n = 0$, for $u,v$ as in the theorem. Indeed, simply search for a common integer zero $n \in \ZZ$, and in parallel search for integer $m \geq 1$ satisfying $(u_n,v_n) \not\equiv (0,0) \bmod m$; this must terminate by \cref{thm:sim-local-global-intro}. Although it is known that this problem is decidable, subject to the $p$-adic Schanuel Conjecture \cite{bacik_p-adic_conference_2026}, this yields an alternative, simpler algorithm. However, more importantly, we believe this result should serve as a stepping-stone towards proving \cref{conj:skolem-conj}, which, as mentioned earlier, would provide a route towards proving decidability of the Skolem Problem. To this end, we pose some conjectures in \cref{sec:conclusion} inspired by the results proven in this paper, with a heuristic argument in support.

We shall now say a little more about the phenomena underlying \cref{ex:bad-TRZ-example} and \cref{thm:sim-local-global-intro}. The failure of \cref{conj:failed-conj} for $u$ in \cref{ex:intro-ex} may be explained by the fact that 0 is in a certain sense, a ``hidden'' zero of $u$. In particular, it is a \emph{twisted rational zero}, a concept introduced by Bilu \emph{et al.} \cite{bilu_twisted_2025}.

\subsection{Twisted rational zeros}
Suppose that $u$ is an LRS satisfying \eqref{eqn:LRS-exp-poly-intro}. Then a rational number $x \in \QQ$ is a twisted rational zero (abbreviated to TRZ) of $u$, if for some definition\footnote{For example, both $2$ and $-2$ are valid definitions of $4^{1/2}$.} of $\lambda_1^x,\dots, \lambda_s^x$, and some roots of unity $\xi_1, \dots, \xi_s$, we have
\begin{align*}
    \sum_{i=1}^s \xi_i P_i(x) \lambda_i^x = 0 \, .
\end{align*}
We further call $x$ a trivial TRZ if $P_i(x) = 0$ for all $i= 1, \dots, s$. The main contribution of \cite{bilu_twisted_2025} was to show a version of the Skolem--Mahler--Lech theorem for TRZs; a non-degenerate LRS has finitely many TRZs.

The concept of twisted rational zeros provides a global explanation for the local behaviour of the LRS $u$ in \cref{ex:intro-ex}. One can see that $0$ is a TRZ of $u$, since $1 +(-1)\cdot 2^0 + (-4)^0 + (-1)\cdot (-8)^0 = 0$, and it is the fact that 0 is a TRZ that is to blame for the failure of \cref{conj:failed-conj}. On the other hand, a given TRZ need not be a $p$-adic zero for every prime $p$. Indeed, although $0$ is a TRZ of $w_n = 1+2^n$, we have $1+2^n \not\equiv 0 \bmod 7$ for all $n$ (indeed, $w_{n+3} \equiv w_n \bmod 7$ and $w_0,w_1,w_2 \not\equiv 0 \bmod 7$), so $w$ in fact has no $7$-adic zeros at all.

The main technical results in this paper are to give a full account of when TRZs are or are not $p$-adic zeros, for each prime $p$. In particular, we prove the following, which answers an open question of Bilu \emph{et al.} \cite[Question 6.7]{bilu_twisted_2025}.

\begin{theorem} \label{thm:avoiding-TRZs-intro}
Let $x \in \QQ$ be a non-trivial TRZ, and $u$ be a non-degenerate LRS.
\begin{enumerate}
    \item If $x \not\in \ZZ$, then there are infinitely many primes $p$ for which $x$ is not a $p$-adic zero.
    \item If $u$ is also not weakly degenerate, $x \in \ZZ$ and $u_x \neq 0$, then there are infinitely many primes $p$ for which $x$ is not a $p$-adic zero.
\end{enumerate}
\end{theorem}
In fact, \cref{thm:avoiding-TRZs-intro} may be generalised to deal with multiple TRZs of multiple LRS simultaneously, see \cref{thm:avoiding_TRZ}. This result is used in the proof of \cref{thm:sim-local-global-intro} -- the fact that the integer $m$ in \cref{thm:sim-local-global-intro} may be taken as a power of two primes corresponds to the fact that one prime is needed to deal with integer TRZs, and another is needed for non-integer TRZs.

On the other hand, we also give a complete answer to \cite[Question 1.7]{bilu_twisted_2025} on when TRZs \emph{are} $p$-adic zeros. For this, we say a TRZ $x$ of an LRS written $u_n = P(n,\mu_1^n, \dots, \mu_r^n)$ is \emph{compatible} if there are roots of unity $\xi_1, \dots, \xi_r$ such that $P(x,\xi_1 \mu_1^x, \dots, \xi_r \mu_r^x) = 0$.
\begin{theorem} \label{thm:detecting-TRZ-intro}
Any non-trivial TRZ of a non-weakly-degenerate LRS is a $p$-adic zero for infinitely many primes $p$ if and only if it is compatible. 
\end{theorem}

Finally, in \cref{sec:conclusion}, we make some remarks on a possible route to making Theorems \ref{thm:avoiding-TRZs-intro}, \ref{thm:detecting-TRZ-intro} and \ref{thm:sim-local-global-intro} effective.

\section{Preliminaries} \label{sec:prelims}
\subsection{\texorpdfstring{$p$}{p}-adic numbers}
We briefly recall relevant notions about $p$-adic numbers. More details may be found in algebraic number theory textbooks such as \cite{neukirch_algebraic_1999}.

Given a prime number $p$, every non-zero rational number $x \in \QQ$ may be written as $x = \frac{a}{b}p^r$ for some integers $a,b$ coprime to $p$, with $b$ non-zero, and $r \in \ZZ$. We define the valuation $v_p(x) = r$, and $v_p(0) = \infty$.  From this derivation we see that $v_p$ satisfies the ultrametric inequality
$v_p(x+y) \geq \min(v_p(x),v_p(y))$ for all $x,y \in \mathbb Q$.

We define an absolute value on $\QQ$ by $|x|_p = p^{-v_p(x)}$.  The ultrametric inequality on $v_p$ translates to the 
strong triangle inequality: $|x+y|_p \leq \max(|x|_p,|y|_p)$ for all $x,y \in \mathbb Q$.
Define the set of $p$-adic numbers $\QQ_p$ as the completion of $\QQ$ with respect to $|\cdot|_p$. If one completes $\ZZ$ with respect to $|\cdot|_p$ then one gets the $p$-adic integers $\ZZ_p$, which may also be defined as the unit disc in $\QQ_p$:
\begin{align*}
\ZZ_p = \{x \in \QQ_p : |x|_p \leq 1\} = \{x \in \QQ_p : v_p(x) \geq 0 \} \, .
\end{align*}
A key property is that any $p$-adic integer $x \in \ZZ_p$ may be written as a power series $x = \sum_{n=0}^\infty a_n p^n$ where $a_n \in \{0,\dots, p-1\}$.

\subsection{Algebraic extensions of \texorpdfstring{$\QQ_p$}{Qp}}
We will require a generalisation of the above discussion to deal with algebraic numbers lying outside of $\ZZ_p$. Given a number field $\KK$, let $\O_\KK$ denote its ring of integers. Previous ideas of factorising to define a valuation $v_p$ do not carry over directly as $\O_\KK$ is no longer a unique factorisation domain, instead one generalises to consider factorisations of (fractional) ideals. Define a \emph{fractional ideal} of $\KK$ to be a non-zero finitely generated $\O_\KK$-submodule of $\KK$; equivalently, $I$ is a fractional ideal if and only if there is $c \in \KK$ such that $cI \subseteq \O_\KK$ is an ideal of $\O_\KK$. Any fractional ideal has a unique factorisation 
\begin{align*}
    I = \prod_{i=1}^t \pp_i^{n_i}
\end{align*}
where $t \in \NN$ and each $\pp \subseteq \O_\KK$ is a prime ideal, $n_i \in \ZZ$ \cite[p. 22]{neukirch_algebraic_1999}. If $I \subseteq \O_\KK$ is an ideal then $n_i \geq 0$. Define the valuation $v_\pp$ by $v_\pp(a) = n$ if the exponent of $\pp$ in the prime factorisation of the ideal $a\O_\KK$ is $n$, and $v_\pp(0) = \infty$. Any prime ideal $\pp$ has $v_\pp(p) > 0$ for a unique integer prime $p \in \NN$. We say $\pp$ ``divides'' or ``lies above'' $p$. Given a rational prime $p \in \ZZ$, we say $p$ ramifies in $\KK$ if $v_\pp(p) =e > 1$ for some prime ideal $\pp \subseteq \O_\KK$ lying above $p$; call $e = e_{\pp/p}$ the \emph{ramification index}. Define the \emph{residue field degree $f = f_{\pp/p}$} by $f = \left[\faktor{\O_\KK}{\pp}:\faktor{\ZZ}{p\ZZ}\right]$. Define the $\pp$-adic absolute value by $|x|_\pp = p^{-v_\pp(x)/e_{\pp/p}}$.\footnote{There are many conventions for the normalization of $\pp$-adic absolute values as they are equivalent and induce the same completions.} Denote the completion of $\KK$ with respect to $|\cdot|_\pp$ by $\KK_\pp$. The field $\KK_\pp$ is an algebraic extension of $\QQ_p$. Let the valuation ring be denoted by 
\begin{align*}
    \O_\pp = \{x \in \KK_\pp: v_\pp(x) \geq 0\}\, .
\end{align*}
If $S$ is a finite set of prime ideals of $\KK$, then we define the ring of $S$-integers to be
\begin{align*}
    \O_S = \O_{\KK,S} = \{x \in \KK : v_\pp(x) \geq 0 \ \text{for all prime ideals }  \pp \not\in S \} \, .
\end{align*}

\subsection{\texorpdfstring{$p$}{p}-adic exponential and logarithm}
Given a number field $\KK$ and prime ideal $\pp \subseteq \O_\KK$, we may define the $\pp$-adic exponential and logarithm functions by their series representations
\begin{align*}
    \exp(x) = \sum_{n=0}^\infty \frac{x^n}{n!} \, , \qquad \log(x) = \sum_{n=1}^\infty (-1)^{n+1} \frac{(x-1)^n}{n} \, ,
\end{align*}
where $\exp$ converges whenever $v_\pp(x) > \frac{e_{\pp/p}}{p-1}$, and $\log$ converges whenever $v_\pp(x-1) > 0$. From the series definitions, it is clear that if $v_\pp(y-1) > \frac{e_{\pp/p}}{p-1}$, then $v_\pp(\log(y)) = v_\pp(y-1) > \frac{e_{\pp/p}}{p-1}$, and therefore $\exp(x \log(y))$ converges for all $x \in \O_\pp$, and defines an analytic function in $x$. If $\exp(y)$ converges, then again from the series definition one sees that we always have $\exp(y) \equiv 1 \bmod \pp$. The usual properties hold, that $\exp(x+y) = \exp(x)\exp(y)$, $\log(xy) = \log(x)+\log(y)$, $\exp(\log(x)) = x$, and $\log(\exp(x)) = x$ whenever the quantities converge.

\subsection{Multiplicative groups and reductions of algebraic numbers}
By the structure theorem for finitely generated abelian groups, any finitely generated multiplicative subgroup $G \leq \KK^\times$ has $G \cong \ZZ^r \times \Gamma$, where $r$ is the \emph{rank} of $G$ and $\Gamma$ is a finite group. We say $G$ is \emph{torsion-free} if $\Gamma$ is trivial. 

For $a \in \KK$ and prime ideal $\pp \subseteq \O_\KK$ such that $v_\pp(a) = 0$, we define\footnote{Caution: other sources sometimes write $\ord_\pp(a)$ to mean the $\pp$-adic valuation of $a$.} $\ord_\pp(a)$ to be the smallest positive integer satisfying $a^{\ord_\pp(a)} \equiv 1 \bmod \pp$. For a multiplicative subgroup $G \leq \KK^\times$, suppose $v_\pp(a) = 0$ for all $a \in G$, then we may define $\overline{G}$ to be the subgroup of $\left(\faktor{\O_\KK}{\pp} \right)^\times$ given by the reduction of $G$ mod $\pp$. We write $\ord_\pp(G) \coloneq |\overline G|$.

Given a prime ideal $\pp \subseteq \O_\KK$, define the norm to be $N(\pp) =  \left| \faktor{\O_\KK}{\pp} \right|$. Recall that, given a set $A$ of prime ideals of $\O_\KK$, the Dirichlet density of $A$ is
\begin{align*}
    \lim_{s \to 1^+} \frac{\sum_{\pp \in A} N(\pp)^{-s}}{\sum_{\pp \subseteq \O_\KK} N(\pp)^{-s} } \, .
\end{align*}
We may also define the natural density
\begin{align*}
    \lim_{x \to \infty} \frac{\# \{\pp \in A \mid N(\pp) \leq x \}}{\#\{\pp \subseteq \O_\KK \mid N(\pp) \leq x \}} \, ,
\end{align*}
and when both exist, the Dirichlet density of $A$ equals the natural density of $A$~\cite[Section~13]{neukirch_algebraic_1999}. Note that the existence of the natural density implies the existence of the Dirichlet density.

A key part of the proofs of our main results are on prescribing certain conditions on orders of algebraic numbers when reduced modulo a prime. Such considerations are related to Artin's Conjecture on primitive roots (see the survey \cite{Moree_2012}). We state two results of Perucca and Sgobba, based on Kummer theory and Chebotarev's Density Theorem. The following is a special case of \cite[Corollary 14]{perucca_2009}, obtained by setting $G_i(K) = K^\times$ in the original notation.
\begin{theorem} \label{thm:perucca}
Let $\KK$ be a number field and for every $i=1, \dots, n$, let $R_i \in \KK^\times$ be an algebraic number that is not a root of unity. Then for each integer $m > 0$ there exists a positive Dirichlet density of primes $\pp \subseteq \O_{\KK}$ such that for each $i = 1 , \dots, n$, we have $\ord_\pp(R_i)$ is a multiple of $m$.
\end{theorem}

We also make use of the following, which is \cite[Theorem 21]{sgobba_divisibility_2023}, obtained by setting $C$ to $\mathrm{Gal}(F/K)$. In fact, they give explicit estimates for the density and an effective version, though we omit this here, as we do not need it. 

\begin{theorem} \label{thm:sgobba}
Let $\KK$ be a number field and let $G$ be a finitely generated and torsion-free subgroup of $\KK^\times$ of positive rank. Let $\mathbb{F}/\KK$ be a finite Galois extension. Consider a finite set $L$ of prime numbers $\ell$, and for each of them fix a non-negative integer $a_\ell$. Consider the set of primes of $\KK$ given by
\begin{align*}
    \mathcal V = \left \{ \pp \subseteq \O_\KK : v_\pp(g) = 0 \ \forall g \in G \, , \, \pp \textnormal{ is unramified in } \mathbb{F}\, , \textnormal{ and } v_\ell(\ord_\pp(G)) = a_\ell \ \forall \ell \in L \right\} \, .
\end{align*}
Then $\mathcal V$ has positive natural density. 
\end{theorem}

\subsection{Linear Recurrence Sequences} \label{sec:prelims:LRS}
An LRS is a sequence $\langle u_n \rangle_{n=0}^\infty$ of algebraic numbers satisfying a recurrence relation
\begin{align*}
    u_{n+d} = a_{d-1} u_{n+d-1} + \dots + a_0 u_n \, ,
\end{align*}
where $a_0, \dots, a_{d-1} \in \Alg$ and $a_0 \neq 0$. If $\KK$ is a number field, and $u_n \in \KK$ for all $n \in \ZZ$, and $a_0, \dots, a_{d-1} \in \KK$, then we call $u$ a $\KK$-LRS. Any LRS $u$ has an exponential polynomial form given by
\begin{align} \label{eqn:LRS-exp-poly}
    u_n = \sum_{i=1}^s P_i(n) \lambda_i^n \, ,
\end{align}
where $P_i \in \Alg[X]$ and $\lambda_i \in \Alg$ for each $i=1,\dots, s$. Let $\LL$ be a number field containing each $\lambda_i$ and the coefficients of $P_i$ for $i=1, \dots, s$.

\begin{definition}
We say that an LRS $u$ is degenerate if $\lambda_i/\lambda_j$ is a root of unity for some $i \neq j$. We say $u$ is weakly degenerate if $G  = \langle \lambda_1, \dots, \lambda_s \rangle \subseteq \LL^\times$ is not torsion-free. 
\end{definition}
A degenerate sequence is obviously weakly degenerate, but the converse need not hold. For example $u_n = 2^n + (-1)^n$ is non-degenerate but is weakly degenerate. 

If $G = \langle \lambda_1, \dots, \lambda_s \rangle \cong \ZZ^r \times \Gamma$, where $\Gamma$ is a finite group of order $M$, then $\langle \lambda_1^M,\dots, \lambda_s^M \rangle \cong \ZZ^r \times M\Gamma \cong \ZZ^r$. Therefore, every subsequence $\langle u_{Mn+\ell} \rangle_{n=0}^\infty$ for $0 \leq \ell \leq M-1$ is not weakly degenerate. Note that it is possible to compute all multiplicative relations between $\lambda_1, \dots, \lambda_s$~\cite{combot_computing_2025} --- that is, the set of $a_1, \dots, a_s \in \ZZ$ such that $\lambda_1^{a_1} \cdots \lambda_s^{a_s} = 1$. Therefore, one may effectively compute multiplicatively independent algebraic numbers $\mu_1, \dots, \mu_r \in \Alg$ such that each $\lambda_i$ may be written as a product $\lambda_i = \zeta_i \prod_{j=1}^r \mu_j^{m_{i,j}}$, for some $m_{i,j} \in \ZZ$ and some root of unity $\zeta_i$. Hence, the order $M$ of $\Gamma$ may be computed as $M = \mathrm{lcm}(\ord(\zeta_1),\dots,\ord(\zeta_s))$ and thus each subsequence $u_{Mn + \ell}$ is not weakly degenerate, and has a computable representation as
\begin{align*}
    u_{Mn+\ell} = Q_\ell(n,\mu_1^{Mn},\dots,\mu_r^{Mn}) \, ,
\end{align*}
for some Laurent polynomial $Q_\ell \in \Alg[X_0^{\pm 1}, \dots, X_r^{\pm 1}]$.

Recall \cite{bilu_twisted_2025} that a \emph{twisted rational zero} of an LRS $u$ given by \eqref{eqn:LRS-exp-poly} is a rational number $x \in \QQ$ if for some definition of $\lambda_1^x, \dots, \lambda_s^x$, we have roots of unity $\xi_1, \dots, \xi_s$ such that
\begin{align} \label{eqn:TRZ-vanish}
    \sum_{i=1}^s \xi_i P_i(x) \lambda_i^x = 0 \, .
\end{align}
The TRZ $x$ is \emph{trivial} if $P_i(x) = 0 $ for all $1 \leq i \leq s$ and non-trivial otherwise. If a non-weakly-degenerate LRS is given by $u_n = P(n,\mu_1^n, \dots, \mu_r^n)$ for Laurent polynomial $P$, then we say a TRZ $x \in \QQ$ is \emph{compatible} with $u$ if there are roots of unity $\xi_1, \dots, \xi_r$ such that for some definition of $\mu_1^x,\dots,\mu_r^x$, we have $P(x,\xi_1\mu_1^x,\dots, \xi_r \mu_r^x) = 0$.

We call a tuple $(x, \xi_1, \dots, \xi_s)$ a \emph{vanishing tuple} for $u$, if \eqref{eqn:TRZ-vanish} holds. We also say that $(x, \xi_1, \dots, \xi_s)$ is associated to $x$.

We further call $(x,\xi_1, \dots, \xi_s)$ \emph{primitive} if \eqref{eqn:TRZ-vanish} holds, $\xi_s = 1$, and no proper sub-sum vanishes; that is, there is no $\varnothing \neq I \subsetneq \{1, \dots, s\}$ such that $\sum_{i \in I} \xi_i P_i(q) \lambda_i^q = 0$.\footnote{This borrows terminology from \cite{bilu_twisted_2025}, though we make primitiveness a property of the TRZ \emph{and} the roots of unity chosen, whereas primitiveness is only defined as a property of the TRZ in \cite{bilu_twisted_2025}, which we find to be a little misleading. This is because different choices of roots of unity can make the TRZ primitive or not. For example, consider 0 as a TRZ of $u_n = 1 \cdot 1^n - 2^n +(1+i)\cdot 3^n - (1+i) \cdot 4^n$. Then the tuple $(0,1,1,1,1)$ is not primitive as we have the vanishing sub-sum $1 \cdot 1^0 - 2^0 = 0$, but the tuple $(0,1,-1,i,1)$ is primitive.} Since for any vanishing tuple we may scale the roots of unity by some root of unity $\zeta$ to get another vanishing tuple, enforcing $\xi_s =1 $ for primitive tuples ensures they are normalised in a canonical way.

The following is a combination of \cite[Theorem 1.9, Corollary 4.6]{bilu_twisted_2025}, which may be regarded as an analogue of the Skolem--Mahler--Lech theorem for twisted rational zeros.
\begin{theorem} \label{thm:finitely-many-primitive-tuples}
Let $u$ be a non-degenerate LRS. Then
\begin{enumerate}
    \item $u$ admits finitely many primitive vanishing tuples, and
    \item $u$ admits finitely many TRZs.
\end{enumerate}
\end{theorem}

\subsection{\texorpdfstring{$p$}{p}-adic zeros of LRS}
Let $u$ be a $\KK$-LRS satisfying \eqref{eqn:LRS-exp-poly} and let $\LL$ be a number field containing $\lambda_i$ and the coefficients of each $P_i$ for $i=1,\dots, s$. Let $S$ be a finite set of prime ideals of $\O_\LL$ such that the ring of $S$-integers $\O_{\LL,S} = \O_S$ contains each $\lambda_i,\lambda_i^{-1}$ and the coefficients of each $P_i$ for $i=1, \dots, s$. Let $\PP \subseteq \O_\LL$ be some prime ideal lying above $\pp \subseteq \O_\KK$ and lying above rational prime $p \in \ZZ$; we say $\PP$ is an \emph{admissible} prime for $u$ if $\PP \not\in S$. In particular, if $\PP$ is admissible, then $v_\PP(\lambda_i) = 0$ for all $i$, and $v_\pp(u_n) \geq 0$ for all $n \in \ZZ$. We say $\pp \subseteq \O_\KK$ is admissible if it lies below an admissible prime in $\O_\LL$. 

\begin{definition} \label{def:TRZ-p-adic}
Given a number field $\KK$, a $\KK$-LRS $u$ and an admissible prime ideal $\pp \subseteq \O_\KK$ lying above rational prime $p$, we call $x \in \ZZ_p$ a $\pp$-adic zero of $u$ if there is a sequence of integers $n_j \in \ZZ$ such that $v_\pp(n_j-x) \to \infty$ and $v_{\pp}(u_{n_j}) \to \infty$ as $j \to \infty$.\footnote{This differs from the definition of a $\pp$-adic zero in \cite{bacik_p-adic_2026} by a linear transformation of each $\pp$-adic zero.} 
\end{definition}
When a TRZ $x \in \QQ$ is also a $\pp$-adic zero, we say that $x$ is \emph{detected} by $\pp$.

We interpret these definitions in terms of $\pp$-adic interpolations of LRS. Let $\PP \subseteq \O_\LL$ be an admissible prime for $u$ and $N \in \NN$ be some integer for which $v_\PP(\lambda_i^N - 1) > \frac{e_{\PP/p}}{p-1}$ for all $1 \leq i \leq s$. Then $\exp \left(x\log\lambda_i^N \right)$ converges for all $x \in \O_\pp$ and so we may define for each $0 \leq \ell \leq N-1$ the $\ell$-th $\pp$-adic interpolant
\begin{align*}
    F_\ell(x) = \sum_{i=1}^s P_i(Nx+\ell) \lambda_i^\ell \exp \left(x \log \lambda_i^N \right) \, .
\end{align*}
Since $\exp$ and $\log$ are analytic on their domains, $F_\ell : \O_\PP \to \O_\PP$ is analytic. 

\begin{lemma} \label{lem:p-adic-zero-def-equal}
Let $u$ be a $\KK$-LRS and let $\pp \subseteq \O_\KK$ be an admissible prime ideal lying above rational prime $p \in \ZZ$. We have that $x \in \ZZ_p$ is a $\pp$-adic zero of $u$ according to \cref{def:TRZ-p-adic} if and only if there is $y \in \ZZ_p$ and $0 \leq \ell \leq N-1$ such that $F_\ell(y) = 0$, and $Ny+\ell =x $.
\end{lemma}
\begin{proof}
If $x \in \ZZ_p$ is a $\pp$-adic zero of $u$, then there is a sequence of integers $n_j \in \NN$ such that $v_\pp(n_j-x) \to \infty$ and $v_\pp(u_{n_j}) \to \infty$ as $j \to \infty$. By the pigeonhole principle there must be $0 \leq \ell \leq N-1$ such that there are infinitely many integers $n_j$ of the form $Nm_j + \ell$. Then $v_\pp(u_{Nm_j + \ell}) = v_\pp(F_\ell(m_j)) \to \infty$ as $j \to \infty$, and $m_j \to y \coloneq \frac{x-\ell}{N}$ in $p$-adic absolute value as $j \to \infty$, so by continuity of $F_\ell$, we have $F_\ell(y) = 0$ and $Ny + \ell =x$. 

For the other direction, if $F_\ell(y) = 0$ and $Ny+\ell = x$, then take any sequence of integers $m_j$ such that $v_\pp(m_j - y) \to \infty$, then we have $v_\pp(Nm_j+\ell - x) \to \infty$ as $j\to \infty$, and $v_\pp(F_\ell(m_j)) = v_\pp(u_{Nm_j+\ell}) \to \infty$ as $j \to \infty$ by continuity and the fact that $F_\ell(y) = 0$. 
\end{proof}

Note that since the lemma holds for any value of $N$ for which $v_\PP(\lambda_i^N-1) > \frac{e_{\PP/p}}{p-1}$, we are free to choose the most convenient such value of $N$ in applications. The following is an easy consequence of \cref{lem:p-adic-zero-def-equal}, though it is also proved in \cite[Theorem 3.5]{bilu_twisted_2025}
\begin{lemma} \label{lem:rat-zero-is-TRZ}
If $u$ is a $\KK$-LRS, $\pp \subseteq \O_\KK$ is an admissible prime ideal, and $x \in \QQ$ is a $\pp$-adic zero, then $x$ is a TRZ.
\end{lemma}
Finally, it is useful to note that the definition of a $\pp$-adic zero is stable upon taking extensions of number fields.

\begin{lemma} \label{lem:pp-adic-to-PP-adic}
Let $u$ be a $\KK$-LRS, let $\pp \subseteq \O_\KK$ be an admissible prime ideal and let $x \in \ZZ_p$. Let $\LL \supseteq \KK$ be some finite extension of $\KK$. The following are equivalent
\begin{enumerate}
    \item $x$ is a $\pp$-adic zero of $u$;
    \item $x$ is a $\PP$-adic zero of $u$ for some prime ideal $\PP \subseteq \O_\LL$ lying above $\pp$;
    \item $x$ is a $\PP$-adic zero of $u$ for all prime ideals $\PP \subseteq \O_\LL$ lying above $\pp$.
\end{enumerate}
\end{lemma}
\begin{proof}
According to \cref{def:TRZ-p-adic}, $x \in \ZZ_p$ is a $\pp$-adic zero if and only if there is a sequence of integers $n_j \in \ZZ$ such that $v_\pp(x- n_j) \to \infty$ and $v_\pp(u_{n_j}) \to \infty$ as $j \to \infty$. For every $\PP \subseteq \O_\LL$ lying above $\pp$, and every $a \in \KK_\pp$, we have $v_\PP(a) = e_{\PP/\pp} v_\pp(a)$, so $v_\pp(x-n_j) \to \infty$ and $v_\pp(u_{n_j}) \to \infty$ if and only if $v_\PP(x-n_j) \to \infty$ and $v_\PP(u_{n_j}) \to \infty$. This is exactly the definition of $x$ being a $\PP$-adic zero. 
\end{proof}

\section{Characterising detection of TRZs by primes} \label{sec:main-sec}
\subsection{Main theorems}
In this section we prove two theorems on when TRZs are $\pp$-adic zeros. \cref{thm:avoiding_TRZ} proves that for every TRZ $y$ of a non-weakly-degenerate sequence that is not a genuine integer zero, there are infinitely many primes $\pp \subseteq \O_\KK$ for which $y$ is not a $\pp$-adic zero, while \cref{thm:detecting-TRZ} proves there are infinitely many primes for which $y$ is a $\pp$-adic zero, if $y$ is compatible. Note that an integer zero of an LRS is a $\pp$-adic zero for every admissible prime $\pp$, trivially from \cref{def:TRZ-p-adic}.

We begin with a useful lemma.

\begin{lemma} \label{lem:root-of-unity-unique}
Let $a/b \in \QQ$ with $\gcd(a,b) = 1$. Let $\lambda \in \KK$ be an algebraic number and $\theta \in \KK$ be such that $\theta^b = \lambda$. Suppose $\pp \subseteq \O_\KK$ is a prime ideal lying above rational prime $p \in \ZZ$ such that $v_\pp(\lambda-1) > \frac{e_{\pp/p}}{p-1}$ and $v_\pp(b) = 0$. Then $\exp(\frac{a}{b} \log \lambda) = \xi \theta^a$ where $\xi$ is some $b$-th root of unity. Furthermore, if $p$ is unramified in $\KK$ and $p>2$, then $\xi$ is the unique root of unity in $\KK_\pp$ with $\xi \theta^a \equiv 1 \bmod \pp$.
\end{lemma}
\begin{proof}
We have $\exp(\frac{a}{b} \log \lambda)^b = \exp(a \log \lambda) = \lambda^a$ so $\exp(\frac{a}{b} \log \lambda)$ is some $b$-th root of $\lambda^a$, meaning $\exp(\frac{a}{b} \log \lambda) = \xi \theta^a$ for some $b$-th root of unity $\xi$.

Furthermore, if $p$ is unramified in $\KK$, then $\KK_\pp/\QQ_p$ is an unramified extension. Suppose there are two roots of unity $\xi, \xi' \in \KK_\pp$ such that $\xi \theta^a \equiv \xi' \theta^a \equiv 1 \bmod \pp$. Then $\xi/\xi' \equiv 1 \bmod \pp$. Suppose that $\zeta \coloneq \xi/\xi' \neq 1$, then it is a primitive $p^\ell m$-th root of unity, for some integers $\ell \geq 0$ and $m\geq 1$ with $\gcd(m,p) = 1$. Note that $\zeta^{p^\ell}$ satisfies the polynomial $X^m-1$, which is separable over $\faktor{\ZZ}{p\ZZ}$, so 1 is the unique $m$-th root of unity which reduces to 1 in $\faktor{\O_\pp}{\pp}$. Thus, since $\zeta^{p^\ell} \equiv 1 \bmod \pp$, we have $\zeta^{p^\ell} = 1$, so $m=1$. But then $\QQ_p(\zeta)/\QQ_p$ is ramified \cite[Proposition 7.13]{neukirch_algebraic_1999}, since $\zeta$ is a $p$-power root of unity, so $\KK_\pp/\QQ_p$ is ramified, which contradicts our assumption. We conclude that $\xi$ is the unique root of unity in $\KK_\pp$ with $\xi \theta^a \equiv 1 \bmod \pp$.
\end{proof}

We now prove our theorem on primes not detecting TRZs. The general idea is that through \cref{lem:root-of-unity-unique}, a TRZ of $u$ is a $\pp$-adic zero if and only if a finite number of congruences of the form $\xi \theta^a \equiv 1 \bmod \pp$ are satisfied. Reducing the problem to congruence conditions in this way allows the application of \cref{thm:perucca} and \cref{thm:sgobba}. In fact, for our later application in the proof of \cref{thm:sim-local-global}, we prove a general form, with several LRS and several TRZs.

\begin{theorem} \label{thm:avoiding_TRZ}
Let $U$ be a finite set of non-degenerate $\KK$-LRS, and for each $u \in U$, let $Y_u$ be the set of non-trivial TRZs of $u$ that are not integer zeros of $u$. Let $Y = \bigcup_{u \in U} Y_u$. Let $Y_{u,1}= Y_u \setminus \ZZ$, and $Y_{u,2} = Y_u \cap \ZZ = Y\setminus Y_{u,1}$. Define also $Y_j = \bigcup_{u \in U} Y_{u,j}$ for $j=1,2$. Then:
\begin{enumerate}
    \item The set of prime ideals $\pp \subseteq \O_\KK$ such that, for each $u \in U$, no $y \in Y_{u,1}$ is a $\pp$-adic zero of $u$, has positive Dirichlet density.
    \item If moreover, the multiplicative group generated by the characteristic roots of all $u \in U$ is torsion-free, then the set of prime ideals $\pp \subseteq \O_\KK$ such that, for each $u \in U$, no $y \in Y_{u,2}$ is a $\pp$-adic zero of $u$, has positive Dirichlet density.
\end{enumerate}
\end{theorem}

\begin{proof}[Proof of \cref{thm:avoiding_TRZ}]
Let $u \in U$ take the form\footnote{Note that since $U$ is finite, we can write the LRS in this way where $\lambda_i$ do not depend on $u$, by allowing the $P_{i,u}$ to be identically zero.}
\begin{align} \label{eqn:u-def-2}
    u_n = \sum_{i=1}^s P_{u,i}(n) \lambda_i^n \, .
\end{align}
Any prime ideals considered in this proof will be assumed to be admissible for every LRS $u \in U$. Note that only finitely many primes are inadmissible, and discarding finitely many primes does not affect the truth of the result, so we may freely discard the inadmissible primes.

Note that by \cref{thm:finitely-many-primitive-tuples}, $Y$ must be finite. For each denominator $b$ of each TRZ $y \in Y_1$, we fix in advance a definition of $b$-th roots of each $\lambda_i$; let $\theta_{i,b} \in \Alg$ be such that $\theta_{i,b}^b = \lambda_i$. Let $\Theta$ be the set of all $\theta_{i,b}$, for each $1 \leq i \leq s$ and each denominator $b$.

Consider all LRS formed by sub-sums of $u \in U$, that is, of the form
\begin{align*}
    u_n^{(I)} = \sum_{i \in I} P_{u,i}(n) \lambda_i^n \, ,
\end{align*}
where $\varnothing \neq I \subseteq \{1, \dots, s\}$. By \cref{thm:finitely-many-primitive-tuples}, each $u_n^{(I)}$ has finitely many primitive vanishing tuples associated to each TRZ $y \in Y$. Let $\Xi$ be the set of all roots of unity appearing in a primitive vanishing tuple of some $u_n^{(I)}$, ranging over $u \in U$ and $\varnothing \neq I \subseteq \{1, \dots, s\}$. Note that $\Xi$ is finite, since there are finitely many subsets $I \subseteq \{1,\dots, s\}$. 

Define $\LL = \KK(\Theta , \Xi)$. By \cref{lem:pp-adic-to-PP-adic}, it suffices to prove $x$ is not a $\PP$-adic zero of any $u \in U$ for a positive density of prime ideals $\PP \subseteq \O_\LL$. Finally, let $N > 0$ be any integer for which $v_\PP(\lambda_i^N -1 ) > \frac{e_{\PP/p}}{p-1}$ for all $1 \leq i \leq s$. We denote the $\ell$-th $\PP$-adic interpolant of $u$ by $F_{u,\ell}$ and the $\ell$-th $\PP$-adic interpolant of $u^{(I)}$ by $F^{(I)}_{u,\ell}$. 

Given a vanishing tuple $(y, \xi_1, \dots, \xi_s)$ for $u$, if we have disjoint subsets $I_0,I_1, \dots, I_t \subseteq \{1, \dots, s\}$ such that $I_0 \cup I_1 \cup \dots \cup I_t = \{1, \dots, s\}$, we have $P_{u,i}(y) = 0$ for all $i \in I_0$, and each tuple $(y, (\xi_i)_{i \in I_j})$ is a primitive vanishing tuple of $u_n^{(I_j)}$ for $j=1 ,\dots, t$, then we say $(y, \xi_1, \dots, \xi_s)$ has a primitive decomposition given by $I_0, I_1, \dots, I_t$. We denote the largest element of a subset $I_j$ by $[I_j]$. 

We need the following lemma.
\begin{lemma} \label{lem:divisibility-condition}
Write $y = a/b$, where $a \in \ZZ$, $b \in \NN$, and $\gcd(a,b) = 1$. If $y \in Y_u$ is a $\PP$-adic zero of some $u \in U$, then there is a vanishing tuple $(y, \xi_1, \dots, \xi_s)$ with a primitive decomposition given by subsets $I_0, I_1, \dots I_t$, and there exists $0 \leq \ell \leq N-1$ such that for each $j=1, \dots, t$ and each $i \in I_j$, we have
\begin{align*}
    \ord_\PP \left( \frac{\theta_{i,b}}{\theta_{[I_j],b}} \right) \mid \ord(\xi_{i})(a-b\ell) \, .
\end{align*}
Furthermore, if $b = 1$ and
\begin{align*}
    \ord_\PP \left( \frac{\theta_{i}}{\theta_{[I_j]}} \right) \mid (a-\ell) \, .
\end{align*}
for all $j=1, \dots, t$ and $i \in I_j$, then $u_y = u_a = 0$.
\end{lemma}
\begin{proof}
If $y$ is a $\PP$-adic zero of $u$ then by \cref{lem:p-adic-zero-def-equal} we have $0 \leq \ell \leq N-1$ such that $\frac{a-b\ell}{bN} \in \ZZ_p$, where $p \in \ZZ$ is the rational prime lying below $\PP$, and
\begin{align*}
    F_{u,\ell} \left(\frac{a-b\ell}{bN} \right) &= \sum_{i=1}^s P_{u,i}(y) \lambda_i^{\ell} \exp \left( \frac{a-b\ell}{bN} \log \lambda_i^N \right) = 0 \, .
\end{align*}
We split the sum into primitive sub-sums; there are disjoint $I_0,I_1, \dots, I_t \subseteq \{1, \dots, s\}$ with $I_0 \cup I_1 \cup \dots \cup I_t = \{1, \dots, s\}$ such that $P_i(y) = 0$ for all $i \in I_0$, and for each $j=1 , \dots, t$ and any proper subset $J \subsetneq I_j$ we have
\begin{align*}
    \sum_{i \in I_j} P_{u,i}(x) \lambda_i^{\ell} \exp \left( \frac{a-b\ell}{bN} \log \lambda_i^N \right) =  0 \, , \qquad \sum_{i \in J} P_{u,i}(x) \lambda_i^{\ell} \exp \left( \frac{a-b\ell}{bN} \log \lambda_i^N \right) \neq 0 \, .
\end{align*}

Since $\frac{a-b\ell}{bN} \in \ZZ_p$, when written in lowest terms, the denominator has valuation zero with respect to $p$. Therefore, we may apply \cref{lem:root-of-unity-unique} to get that
\begin{align*}
    \exp \left(\frac{a-b\ell}{bN} \log \left(  \lambda_{i}^N \right) \right) = \zeta_{i} \theta_{i,b}^{a-b\ell} \, ,
\end{align*}
where $\zeta_{i}$ is a root of unity with $\zeta_{i}^{bN} = 1$. Therefore, this gives rise to a primitive vanishing tuple $\left( y, \left( \zeta_i/\zeta_{[I_j]} \right)_{i \in I_j} \right)$ of $u^{(I_j)}$. Let $\xi_i = \zeta_i/\zeta_{[I_j]}$. 

Now note that since 
\begin{align*}
\exp \left(\frac{a-b\ell}{bN} \log \left(  \lambda_{i}^N \right) \right) \equiv 1 \mod \PP \, ,
\end{align*}
we have $\zeta_i \theta_{i,b}^{a-b\ell} \equiv 1 \bmod \PP$, and therefore
\begin{align*}
    \xi_{i} \left( \frac{\theta_{i,b}}{\theta_{[I_j],b}} \right)^{a-b\ell} \equiv 1 \mod \PP \, ,
\end{align*}
for all $i \in I_j$, and so $\left( \theta_{i,b}/\theta_{[I_j],b} \right)^{(a-b\ell) \ord(\xi_{i})} \equiv 1 \mod \PP$, which implies that 
\begin{align} \label{eqn:divisibility-condition}
    \ord_\PP \left( \frac{\theta_{i,b}}{\theta_{[I_j],b}} \right) \mid \ord(\xi_{i})(a-b\ell) \, ,
\end{align}
and this is what was required to show, for the first part.

For the second part, if $b=1$ and $\ord_\PP \left( \lambda_{i}/\lambda_{[I_j]} \right) \mid (a-\ell)$ for all $j = 1, \dots, t$ and $i \in I_j$, then we have  
\begin{align*}
    &F^{(I_j)}_{u,\ell} \left( \frac{a- \ell}{N} \right) = \lambda_{[I_j]}^\ell\exp \left(\frac{a-\ell}{N} \log \left(  \lambda_{[I_j]}^N \right) \right) \sum_{i \in I_j} P_{u,i}(a) \left(\frac{\lambda_i}{\lambda_{[I_j]}}\right)^\ell \exp \left(\frac{a-\ell}{N} \log \left(  \left(\frac{\lambda_i}{\lambda_{[I_j]}} \right)^N \right) \right) \\
   &=   \lambda_{[I_j]}^\ell\exp \left(\frac{a-\ell}{N} \log \left(  \lambda_{[I_j]}^N \right) \right) \sum_{i \in I_j} P_{u,i}(a) \left(\frac{\lambda_i}{\lambda_{[I_j]}}\right)^\ell \exp \left(\frac{a-\ell}{\ord_\PP(\lambda_i/\lambda_{[I_j]})} \log \left(  \left(\frac{\lambda_i}{\lambda_{[I_j]}} \right)^{\ord_\PP(\lambda_i/\lambda_{[I_j]})} \right) \right) \\
   &=   \lambda_{[I_j]}^\ell\exp \left(\frac{a-\ell}{N} \log \left(  \lambda_{[I_j]}^N \right) \right) \sum_{i \in I_j} P_{u,i}(a) \left(\frac{\lambda_i}{\lambda_{[I_j]}}\right)^\ell \exp \left( \log \left(  \left(\frac{\lambda_i}{\lambda_{[I_j]}} \right)^{a-\ell} \right) \right) \\
   &=  \lambda_{[I_j]}^\ell\exp \left(\frac{a-\ell}{N} \log \left(  \lambda_{[I_j]}^N \right) \right) \sum_{i \in I_j} P_{u,i}(a) \left(\frac{\lambda_i}{\lambda_{[I_j]}}\right)^a \\
   &= \frac{\lambda_{[I_j]}^\ell \exp \left(\frac{a-\ell}{N} \log \left(  \lambda_{[I_j]}^N \right) \right)}{\lambda_{[I_j]}^a} u^{(I_j)}_a  \\
   &= 0 \, ,
\end{align*}
so $u^{(I_j)}_a = 0$ for all $j = 1 , \dots, t$, and therefore $u_a = \sum_{j=1}^t u^{(I_j)}_a = 0$.
\end{proof}
We now show how the divisibility condition in \cref{lem:divisibility-condition} may be avoided for a positive density of primes $\PP$, for TRZs in $Y_1$ and $Y_2$.

\textbf{Case 1:} For each $y \in Y_1$, write $y = a/b$ where $a \in \ZZ$, $b \in \NN$ and $\gcd(a,b) = 1$. Let $B$ be the set of all denominators of all $y \in Y_1$.

Define the integer
\begin{align*}
    m = 1+ \max\{v_q(\ord(\xi)) \mid \xi \in \Xi \text{ and } q \text{ is a prime dividing some }  b \in B \} \, .
\end{align*}
Consider the multiplicative group $\KK^\times$, and consider the finite set of elements 
\begin{align*}
    A = \left\{ \frac{\theta_{i,b}}{\theta_{j,b}} \, \middle| \, \lambda_i,\lambda_j \text{ are both characteristic roots of some $u \in U$,} \text{ and } i\neq j\, , \text{ and } b \in B \right\} \, .
\end{align*}
Since each $u \in U$ is non-degenerate, no element of $A$ is a root of unity. Therefore, by \cref{thm:perucca} applied to the elements of $A$, there is a positive Dirichlet density of primes $\PP \subseteq \O_\LL$ such that for every $i \neq j$, we have $\prod_{\tilde b \in B} \tilde{b}^m \mid \ord_\PP(\theta_{i,b}/\theta_{j,b})$ for every $b \in B$. Now, given any $u \in U$, $y = a/b \in Y_{u,1}$, and any prime $q \mid b$, \cref{lem:divisibility-condition} implies that for such primes $\PP \subseteq \O_\LL$, if $y$ is a $\PP$-adic zero of $u$, then we have $q^m \mid \ord(\xi)(a-b\ell)$ for some $\xi \in \Xi$ and some integer $\ell$. But since $\gcd(a,b) = 1$, $q \nmid (a-b \ell)$; indeed otherwise $q \mid a$ which contradicts $\gcd(a,b)=1$. Therefore, $q^m \mid \ord(\xi)$. This contradicts the definition of $m$, so $y$ is not a $\PP$-adic zero of $u$. 

\textbf{Case 2:} Let $Q$ be the (finite) set of rational primes that divide any $\ord(\xi)$ for $\xi \in \Xi$. Let $G = \langle \lambda_1, \dots, \lambda_s \rangle \subseteq \LL^\times$, then $G$ is torsion-free, so we may apply \cref{thm:sgobba} to conclude that there is a positive density of prime ideals $\PP \subseteq \O_\LL$ such that $v_{q}(\ord_\PP(G)) = 0$ for all $q \in Q$. In particular, for each $1 \leq i, j \leq s$ with $i \neq j$, we have $q \nmid \ord_\PP(\lambda_i/\lambda_j)$ for each $q \in Q$, and hence $\gcd(\ord_\PP(\lambda_i/\lambda_j) , \ord(\xi_i) ) = 1$. 

Suppose that for such a prime $\PP$, there is $u \in U$ for which some $y \in Y_{u,2}$ is a $\PP$-adic zero of $u$. Then by \cref{lem:p-adic-zero-def-equal}, we have a vanishing tuple $(y, \xi_1, \dots, \xi_s)$ with primitive decomposition $I_0,I_1, \dots, I_t$ such that for each $j=1, \dots, t$ and all $i \in I_j$ we have
\begin{align*}
    \ord_\PP \left( \frac{\lambda_{i}}{\lambda_{[I_j]}} \right) \mid \ord(\xi_{i})(a-\ell) \, .
\end{align*}
Furthermore, since $u_y \neq 0$ by assumption, there is some $i,j$ for which $\ord_\PP \left( \lambda_{i}/\lambda_{[I_j]} \right) \nmid (a-\ell) $, so therefore $\gcd \left(\ord_\PP \left( \lambda_{i}/\lambda_{[I_j]} \right) , \ord(\xi_i) \right) > 1$. But this contradicts the choice of $\PP$. So no $y \in Y_{u,2}$ is a $\PP$-adic zero of $u$ for any such prime $\PP \subseteq \O_\LL$.
\end{proof}

\begin{corollary}[Answer to \protect{\cite[Question 6.7]{bilu_twisted_2025}}] \label{cor:avoiding-cor}
Let $u$ be a non-degenerate $\KK$-LRS and $x \in \QQ$. Suppose for all prime ideals $\pp \subseteq \O_\KK$ except a set of zero density, there is a sequence of integers $n_j \in \ZZ$ (depending on $\pp$) such that
\begin{align*}
    v_\pp(n_j - x) \to \infty \, \qquad v_\pp(u_{n_j}) \to \infty \, .
\end{align*}
Then
\begin{enumerate}
    \item if $x \not\in \ZZ$, then $x$ is a trivial TRZ.
    \item if $u$ is also non-weakly-degenerate, then either $x$ is a trivial TRZ, or $x \in \ZZ$ and $u_x = 0$. 
\end{enumerate}
\end{corollary}
\begin{proof}
Note that the hypothesis says exactly that $x$ is a $\pp$-adic zero for all but finitely many prime ideals $\pp \subseteq \O_\KK$, so in particular $x$ must be a TRZ of $u$ by \cref{lem:rat-zero-is-TRZ}. Item 1 simply follows from the contrapositive of item 1 of \cref{thm:avoiding_TRZ}, while item 2 follows from the contrapositive of item 2 of \cref{thm:avoiding_TRZ}, using $U = \{u\}$ and $Y = \{x\}$.
\end{proof}

\begin{remark}
\cref{thm:avoiding_TRZ} item 1 fails for degenerate LRS. A somewhat silly example is $u_n = 1 - (-1)^n$, with the TRZ $1/3$. We have $u_{2n} = 0$ and $u_{2n+1} = 2$ for every $n \in \ZZ$, and these subsequences trivially extend to constant $p$-adic analytic functions $F_0, F_1 : \ZZ_p \to \ZZ_p$ defined by $F_0(n) = u_{2n}$ and $F_1(n) = u_{2n+1} = 2$. Therefore, $F_0(1/6) = 0$, identifying $1/3$ as a $p$-adic zero for every prime $p \geq 5$.
\end{remark}

\begin{remark}
\cref{thm:avoiding_TRZ} part 2 fails for weakly-degenerate LRS, even if they are non-degenerate. See \cref{ex:bad-TRZ-example}.
\end{remark}

\begin{remark}
\cref{thm:avoiding_TRZ} shows that for LRS which are not weakly degenerate, we can pick one single prime ideal $\pp$ for which no integer TRZ (that is not a genuine integer zero of the LRS) is detected as a $\pp$-adic zero, and we can also pick one single prime ideal $\qq$ for which no non-integer TRZs are detected as $\qq$-adic zeros, but there need not exist one prime ideal that avoids detecting all TRZs that are not genuine integer zeros. See \cref{ex:two-TRZs}.
\end{remark}

We now prove our theorem on primes detecting TRZs, answering \cite[Question 1.7]{bilu_twisted_2025}.

\begin{theorem} \label{thm:detecting-TRZ}
Any non-trivial TRZ of a non-weakly-degenerate $\KK$-LRS $u$ is a $\pp$-adic zero for a positive density of primes $\pp \subseteq \O_\KK$ if and only if it is compatible.
\end{theorem}
\begin{proof}
We once again assume any prime ideal considered is admissible for $u$, as this discards only finitely many primes. Let $u$ have characteristic roots $\lambda_1, \dots, \lambda_s$, whose multiplicative group is generated by $\mu_1, \dots, \mu_r$, with
\begin{align*}
    \lambda_i = \mu_1^{m_{i,1}} \cdots \mu_r^{m_{i,r}} \, .
\end{align*}
Let $y \in \QQ$ be a TRZ of $u$, and write $y=a/b$, where $\gcd(a,b) = 1$. First we prove that if $y$ is a $\pp$-adic zero for some prime ideal $\pp$, it is compatible with $u$. Indeed, let $\PP$ be a prime of $\KK(\mu_1, \dots, \mu_r)$, and let $N >0$ be an integer such that $v_\PP(\mu_i^{N}-1) > \frac{e_{\PP/p}}{p-1}$. Let $\theta_i$ be a $b$-th root of $\mu_i$; we have $\theta_i^b = \mu_i$. If $a/b$ is a $\PP$-adic zero, then by \cref{lem:p-adic-zero-def-equal} there is integer $0 \leq \ell \leq N-1$ such that $\frac{a-b\ell}{bN} \in \ZZ_p$, and we have
\begin{align*}
    F_\ell \left(\frac{a-b\ell}{bN} \right) &= \sum_{i=1}^s P_i(a/b) \lambda_i^\ell \exp \left(\frac{a-b\ell}{bN} \log(\lambda_i^N) \right) = 0 \\
    &= \sum_{i=1}^s P_i(a/b) \lambda_i^\ell \prod_{j=1}^r \exp \left(\frac{a-b\ell}{bN} \log(\mu_j^{N}) \right)^{m_{i,j}} \\
    &= \sum_{i=1}^s P_i(a/b) \lambda_i^\ell \prod_{j=1}^r (\xi_j \theta_j^{a-b\ell})^{m_{i,j}} \\
    &= \sum_{i=1}^s P_i(a/b) \prod_{j=1}^r (\xi_j \theta_j^{a})^{m_{i,j}} \, ,
\end{align*}

where $\xi_j$ are some $bN$-th roots of unity, by \cref{lem:root-of-unity-unique}. This is exactly the definition of a compatible TRZ.

Conversely, suppose that $a/b$ is a compatible TRZ. Then there exist roots of unity $\xi_1, \dots, \xi_r$ such that
\begin{align*}
\sum_{i=1}^s P_i(a/b) (\xi_1 \theta_1^a)^{m_{i,1}} \cdots (\xi_r \theta_r^a)^{m_{i,r}} = 0 \, ,
\end{align*}
where $\theta_i^b = \mu_i$. Let $\LL = \KK(\theta_1,\dots,\theta_r, \xi_1,\dots,\xi_r)$, and consider prime ideals $\PP \subseteq \O_\LL$ lying above $\pp \subseteq \O_\KK$ and rational prime $p \in \ZZ$, such that $p>2$ and $p$ is unramified in $\LL$. Note that restricting to $\PP$ lying above unramified $p>2$ excludes finitely many primes. Let $N$ be the smallest positive integer such that $v_\PP(\mu_i^N-1) > \frac{e_{\PP}}{p-1}$ for all $i = 1 , \dots, r$. By \cref{lem:root-of-unity-unique}, for any integer $0 \leq \ell \leq N-1$, and for each $i=1,\dots,r$, there is a root of unity $\zeta_i \in \LL_\PP$ such that
\begin{align} \label{eqn:exp-log-eq}
    \exp \left( \frac{a-b\ell}{bN} \log(\mu_i^N)\right) = \zeta_i \theta_i^{a-b\ell} \, ,
\end{align}
and furthermore $\zeta_i$ is the unique root of unity in $\LL_\PP$ with $\zeta_i \theta_i^{a-b\ell} \equiv 1 \bmod \PP$. Now, $a/b$ is certainly identified as a $\PP$-adic zero for such prime ideal $\PP$ if $\zeta_i = \xi_i$, which by uniqueness of $\zeta_i$ is the case if and only if $\xi_i \theta_i^{a-b\ell} \equiv 1 \bmod \PP$. It suffices therefore, to show there is a positive density of such prime ideals $\PP$ for which there is integer $0 \leq \ell \leq N-1$ with $\xi_i \theta_i^{a-b\ell} \equiv 1 \bmod \PP$ for all $i = 1, \dots, r$. 

Now, let $M = \prod_i \ord(\xi_i)$. If $b=1$, define $c=1$, otherwise if $b>1$, pick $c$ to be an integer such that $\gcd(c,M) =1$ and $c \equiv -a \bmod b$. Indeed, use the Chinese Remainder Theorem to find $c$ satisfying $c \equiv -a \bmod b$ and $c \equiv 1 \bmod q$ for all primes $q \mid M$ such that $q \nmid b$. Then for any prime $q' \mid M$ satisfying $q' \mid b$, the fact that $\gcd(a,b) = 1$ ensures that $q' \nmid c$, so indeed $\gcd(c,M) = 1$ as required.

Now pick $\eta_i$ to be a root of unity such that $\eta_i^c = \xi_i$ for all $i$. Note that $\eta_i \in \LL$, as it is just some power of $\xi_i$, since $\gcd(c,M) =1$. Let $G = \langle 
\theta_1/\eta_1, \dots, \theta_r/\eta_r\rangle \leq \LL^\times$. We prove $G$ is torsion-free; suppose there are integers $n_1, \dots, n_r$ such that $(\theta_1/\eta_1)^{n_1}\cdots (\theta_r/\eta_r)^{n_r} = \zeta$ for some root of unity $\zeta$. This implies that
\begin{align*}
((\theta_1/\eta_1)^{n_1}\cdots (\theta_r/\eta_r)^{n_r})^{bM\ord(\zeta)} = \mu_1^{n_1 M\ord(\zeta)} \cdots \mu_r^{n_r M \ord(\zeta)} = 1 \, ,
\end{align*}
and by multiplicative independence of the $\mu_1, \dots, \mu_r$ this means $n_1 M \ord(\zeta) = \dots = n_r M \ord(\zeta) = 0$ so $n_1 = \dots = n_r = 0$ and $\zeta =1$, so $G$ is indeed torsion-free (and clearly of positive rank, as e.g. $\theta_1/\eta_1$ is not a root of unity).

Therefore, we may apply \cref{thm:sgobba} to conclude that there is a positive density of prime ideals $\PP \subseteq \O_\LL$ such that $\ord_\PP(G)$ is coprime to $bM$, and so $\ord_\PP(\theta_i/\eta_i)$ is coprime to $bM$ for all $i$. Then by Chinese Remainder Theorem, we may find a solution to the congruences
\begin{align*}
    e \equiv 0 \mod  \prod_i \ord_\PP(\theta_i/\eta_i)\, , \quad e \equiv c \mod M\, , \quad e \equiv c \mod b
\end{align*}
and for such an integer $e$, we have $e = -a + b \ell$ for some integer $\ell$, and
\begin{align*}
    \theta_i^{-a+b\ell} &\equiv \theta_i^{e} \mod \PP \\
    &\equiv \eta_i^e \left( \frac{\theta_i}{\eta_i} \right)^e \mod \PP \\
    &\equiv \xi_i \mod \PP \, ,
\end{align*}
and hence $\xi_i \theta_i^{a-b\ell} \equiv 1 \bmod \PP$. Note that if $\ell' \equiv \ell \bmod N$, then $\theta_i^{b\ell} \equiv \theta_i^{b \ell'} \bmod \PP$, so we may pick $0 \leq \ell' \leq N-1$ such that $\ell' \equiv \ell \bmod N$ and so $\xi_i \theta_i^{a-b\ell'} \equiv 1 \bmod \PP$ as required.
\end{proof}

\subsection{Key examples}
We now show by example that our results are best possible; the assumptions of non-degeneracy or non-weak-degeneracy may not be dropped in \cref{thm:avoiding_TRZ} and \cref{thm:detecting-TRZ}. First, a useful lemma to be used in the examples.
\begin{lemma} \label{lem:even-ord}
Let $a \in \ZZ$, and let $p \in \ZZ$ be an odd prime not dividing $a$. If $\ord_p(a)$ is even, then $\exp(\frac{1}{2} \log a^{\ord_p(a)}) = -a^{\ord_p(a)/2}$.
\end{lemma}
\begin{proof}
By \cref{lem:root-of-unity-unique} we have $\exp(\frac{1}{2} \log a^{\ord_p(a)}) = \pm a^{\ord_p(a)/2}$. If $\exp(\frac{1}{2} \log a^{\ord_p(a)}) = a^{\ord_p(a)/2}$, then $a^{\ord_p(a)/2} \equiv 1 \bmod p$, contradicting the definition of $\ord_p(a)$ as the minimal integer for which $a^{\ord_p(a)} \equiv 1 \bmod p$.
\end{proof}

\begin{example} \label{ex:bad-TRZ-example}
Let 
\begin{align*}
    u_n = (1+2^n)(1 + (-4)^n) = 1 + 2^n +(-4)^n + (-8)^n \, .
\end{align*}
Clearly, $u$ is non-degenerate, though it is weakly degenerate as $(-4) \cdot 2^{-2} = -1$, so the multiplicative group $\langle 2, -4, -8\rangle \subseteq \QQ^\times$ has torsion. Furthermore, we observe that $u_0 = 4 \neq 0$.

We show that 0 is a $p$-adic zero of $u$ for every $p>2$. We split into 2 cases.

\textbf{Case 1:} $\ord_p(2)$ is even. Let $N = \ord_p(2)$. Then $(-4)^{N} \equiv (2^{N})^2 \equiv 1 \bmod p$. Also, by \cref{lem:even-ord}, we have $\exp(\frac{1}{2} \log 2^N) = -2^{N/2}$. Consider the $N/2$-th $p$-adic interpolant
\begin{align*}
    F_{N/2}(x) = (1 + 2^{N/2} \exp(x \log 2^N)) ( 1 + (-4)^{N/2} \exp(x \log (-4)^N)) \, ,
\end{align*}
and compute that
\begin{align*}
    F_{N/2}(-1/2) &= (1 + 2^{N/2} \exp(-\frac{1}{2} \log 2^N)) ( 1 + (-4)^{N/2} \exp(-\frac{1}{2} \log (-4)^N)) \\
    &= (1 + 2^{N/2} (-2^{-N/2}))( 1 + (-4)^{N/2} \exp(-\frac{1}{2} \log (-4)^N)) \\
    &= 0 \, .
\end{align*}
Therefore, by \cref{lem:p-adic-zero-def-equal}, this identifies 0 as a $p$-adic zero. 

\textbf{Case 2:} $\ord_p(2)$ is odd. Then $(-4)^{\ord_p(2)} \equiv -(2^{\ord_p(2)})^2 \equiv -1 \bmod p$, so $\ord_p(-4) = 2\ord_p(2)$, so is even. Let $N = \ord_p(-4)$, then by \cref{lem:even-ord}, we have $\exp(\frac{1}{2} \log (-4)^N) = -(-4)^{N/2}$, and therefore we again have $F_{N/2}(-1/2)$, once again proving that 0 is a $p$-adic zero. 
\end{example}

Furthermore, even LRS which are not weakly degenerate may have twisted rational zeros such that at least one is a $p$-adic zero for every prime $p$, despite having no integer zero.

\begin{example} \label{ex:two-TRZs}
Consider the non-weakly-degenerate LRS
\begin{align*}
    u_n = -(4^n+1)(4^n-2)(4^n+2) = 4 + 4\cdot 4^n -16^n -64^n \, .
\end{align*}
There are at least two TRZs of $u$; we have vanishing tuples $(0,1,-1,1,-1)$ for the TRZ 0, and $(1/2,1,1,1,1)$ for the TRZ $1/2$, where we take $4^{1/2} \coloneq 2$. We show that for any prime $p > 2$, either $0$ or $\frac{1}{2}$ arises as a $p$-adic zero of $u$. 

Given prime number $p$, let $N = \ord_p(4)$. Then we have
\begin{align*}
    F_\ell(n) = -(4^\ell\exp(n\log4^N) +1)(4^\ell\exp(n\log 4^N)-2)) (4^\ell\exp(n \log 4^N) +2) \, .
\end{align*}
\textbf{Case 1:} $N$ is even. Then by \cref{lem:even-ord}, we have $\exp(\frac{1}{2} \log(4^N)) = -4^{N/2}$. Therefore we may compute
\begin{align*}
F_{N/2}(- 1/2) = -(-4^{N/2} \cdot 4^{-N/2} +1)(4^{N/2}\cdot 4^{-N/2} -2) (4^\ell \cdot 4^{-N/2} +2) = 0 \, .
\end{align*}
Since $N(-\frac{1}{2}) + \frac{N}{2} = 0$, this identifies $0$ as a $p$-adic zero.

\textbf{Case 2:} $N$ is odd. Then we have
\begin{align*}
    F_{\frac{N+1}{2}}(- 1/2) = -(4^{\frac{N+1}{2}} \exp(-\frac{1}{2} \log4^N) -1)(4^{\frac{N+1}{2}} \exp(-\frac{1}{2} \log 4^N)-2)) (4^{\frac{N+1}{2}} \exp(-\frac{1}{2} \log 4^N) +2) \, .
\end{align*}
Now, $\exp(\frac{1}{2} \log 4^N) = \pm 2^N$, so $4^{\frac{N+1}{2}} \exp(-\frac{1}{2} \log 4^N) = \pm 2$. Either choice of sign makes one of the factors $(4^{\frac{N+1}{2}} \exp(-\frac{1}{2} \log 4^N) - 2)$ or $(4^{\frac{N+1}{2}} \exp(-\frac{1}{2} \log 4^N) +2)$ vanish, so $F_{\frac{N+1}{2}}(-\frac{1}{2}) = 0$, and since $N(-\frac{1}{2}) + \frac{N+1}{2} = \frac{1}{2}$, this identifies $\frac{1}{2}$ as a $p$-adic zero.

This proves the claim. 
\end{example}

\begin{remark}
Recall \cref{conj:skolem-conj} in the case of $\QQ$-LRS. \cref{ex:two-TRZs} shows that even for non-weakly-degenerate LRS, if we strengthen the conjecture to enforce that $m$ is a prime power, then the statement fails.
\end{remark}

\section{A simultaneous exponential local-global principle} \label{sec:local-global}
Given a non-weakly-degenerate $\KK$-LRS $u$, whose characteristic roots are generated by multiplicatively independent algebraic numbers $\mu_1, \dots, \mu_r$ lying in a number field $\LL$, recall that one may find a multivariate Laurent polynomial $P \in \LL[X_0^{\pm 1},\dots,X_r^{\pm 1}]$ such that $u_n = P(n,\mu_1^n, \dots, \mu_r^n)$. If we have two LRS $u,v$, which may be written in the form $u_n = P(n, \mu_1^n, \dots, \mu_r^n)$ and $v_n = Q(n, \mu_1^n, \dots, \mu_r^n)$, where $P,Q \in \LL[X_0^{\pm 1}, \dots, X_r^{\pm 1}]$ are coprime Laurent polynomials, then we say that $u$ and $v$ are coprime. The following was proven in \cite{bacik_p-adic_2026}.\footnote{The particular result is Theorem 11 (ii) of the arXiv version \cite{bacik_p-adic_2026}. The conference version \cite{bacik_p-adic_conference_2026} has an error which the arXiv version corrects.}
\begin{theorem} \label{thm:simultaneous}
Suppose $u,v$ are coprime non-weakly-degenerate $\KK$-LRS. Then, assuming the $p$-adic Schanuel Conjecture, for any prime ideal $\pp$, the only common $\pp$-adic zeros of $u,v$ are rational. 
\end{theorem}
From our previous theorems, we may use this theorem to deduce an exponential local-global principle for simultaneous zeros of coprime LRS. Let $\KK$ be a number field, and let $S$ be a finite set of prime ideals of $\KK$ and coonsider the ring $\O_S$ of $S$-integers of $\KK$.

Suppose $u$ and $v$ are LRS with exponential polynomial forms
\begin{align*}
    u_n = \sum_{i=1}^{s} P_i(n) \lambda_i^n \, , \qquad v_n = \sum_{i=1}^{s} Q_i(n) \lambda_i^n
\end{align*}
where $P_i,Q_i \in \O_S[x]$ (note they can be identically zero) and $\lambda_1, \dots, \lambda_s,\lambda_1^{-1}, \dots, \lambda_s^{-1} \in \O_S$. 
\begin{theorem}[Simultaneous Exponential Local-Global Principle] \label{thm:sim-local-global}
Suppose $u,v$ are coprime non-weakly-degenerate $\KK$-LRS as above. Then, subject to the $p$-adic Schanuel Conjecture, there exists $n \in \ZZ$ with $u_n = v_n = 0$ if and only if for all non-zero ideals $\mathfrak{a}$ of the ring $\O_S$, there exists $n \in \ZZ$ with $u_n \equiv v_n \equiv 0 \mod \mathfrak{a}$.

Furthermore, we may take $\mathfrak a = (\pp \qq)^m$ for some integer $m \geq 1$ and prime ideals $\pp,\qq \subseteq \O_\KK$.
\end{theorem}
\begin{proof}
One direction is trivial; if $u_n = v_n = 0$ then clearly for any ideal $\mathfrak{a}$ of $\O_S$, we have $u_n \equiv v_n \equiv 0 \mod \mathfrak{a}$.  

For the other direction, suppose there is no $n \in \ZZ$ with $u_n = v_n = 0$.

By \cref{thm:simultaneous}, for any prime ideal $\pp \subseteq \O_\KK$, any $\pp$-adic zero shared between $u$ and $v$ is rational, so is in particular, a TRZ, by \cref{lem:rat-zero-is-TRZ}. Any shared TRZ must be a non-trivial TRZ of at least one of $u$ or $v$. Indeed, if $y$ is a trivial TRZ of both $u$ and $v$, then $(n-y)$ is a common factor of $u_n, v_n$, contradicting the hypothesis that $u,v$ are coprime. 

Let $Y_u$ be the set of all non-trivial TRZs of $u$ which are also TRZs of $v$ and not genuine integer zeros of $u$, and $Y_v$ be the analogous set for $v$. Define $Y_{u,1} = Y_u \setminus \ZZ$, $Y_{u,2} = Y_u \cap \ZZ$ and $Y_{v,1} = Y_v \setminus \ZZ$, $Y_{v,2} = Y_v \cap \ZZ$. Finally, define $Y_1 = Y_{u,1} \cup Y_{v,1}$, $Y_2 = Y_{u,2} \cup Y_{v,2}$ and $Y = Y_1 \cup Y_2$. Note that $Y$ is finite by \cref{thm:finitely-many-primitive-tuples}, and contains all the common TRZs of $u$ and $v$.

By \cref{thm:avoiding_TRZ}, there exist infinitely primes $\pp \subseteq \O_\KK$ not lying in $S$ (as $S$ is finite) such that no element of $Y_{u,2}$ is a $\pp$-adic zero of $u$, and no element of $Y_{v,2}$ is a $\pp$-adic zero of $v$. 

Now, since each $y \in Y_2$ is not a $\pp$-adic zero of at least one of $u$ or $v$, this implies that there is a neighbourhood of $y$ on which the minimum valuation of $u$ or $v$ are bounded. Indeed, consider the case that $y \in Y_{u,2}$ (the case $y \in Y_{u,1}$ is symmetric). Then since $y$ is not a $\pp$-adic zero of $u$, there exist integers $r_y,m_y \geq 0$ such that for every $n \in \ZZ$, if $v_\pp(n-y) \geq r_y$, we have $v_\pp(u_n) < m_y$. Indeed, suppose to the contrary that for all $j \in \NN$, there exists integer $n_j \in \ZZ$ with $v_\pp(n_j-y) \geq j$ and $v_\pp(u_{n_j}) \geq j$. Then $v_\pp(n_j-y) \to \infty$ and $v_\pp(u_{n_j}) \to \infty$ as $j \to \infty$; this is exactly the definition of $y$ being a $\pp$-adic zero of $u$, so this gives a contradiction.

Repeat this for every $y \in Y_2$, and let $m = \max\{m_y \mid y \in Y_2\}$ and $r = \max\{ r_y \mid y \in Y_2\}$. Then we conclude that whenever there is some $y \in Y_2$ for which $v_\pp(n-y) \geq r$, we have $\min \{ v_\pp(u_n), v_\pp(v_n)\} < m$.

Now, let $N = p^r$. Then for any $y \in Y_2$, there is a unique $0 \leq \ell_y \leq N-1$ for which there is $n_y \in \ZZ$ with $Nn_y + \ell_y = y$. Now, for each $y \in Y_2$, for every $n \in \ZZ$ we have $N n + \ell_y \equiv y \mod p^r$, so $v_\pp(Nn+\ell_y - y) \geq r$, and therefore $\min\{v_\pp(u_{Nn+\ell_y}) , v_\pp(v_{Nn+\ell_y})\} < m$. 

Thus, we only need to deal with the arithmetic progressions $\langle Nn+\ell \rangle_{n= -\infty}^\infty$ that do not contain any $y \in Y_2$. Let $L$ be the set of integers $0 \leq \ell \leq N-1$ such that the arithmetic progression $\langle Nn+\ell \rangle_{n= -\infty}^\infty$ does not contain any $y \in Y_2$. 

Crucially, for each such $\ell \in L$, each pair of subsequences $\langle u_{Nn+\ell} \rangle_{n=0}^\infty, \langle v_{Nn+\ell} \rangle_{n=0}^\infty$ for $\ell \in L$ has \emph{only} non-integer common TRZs, by definition of $L$. Therefore, by \cref{thm:avoiding_TRZ}, there exists infinitely many prime ideals $\qq \subseteq \O_\KK$ such that every such subsequence $\langle u_{Nn+\ell} \rangle_{n=0}^\infty , \langle v_{Nn+\ell} \rangle_{n=0}^\infty$ for every $\ell \in L$ has no TRZ that is also a $\qq$-adic zero of the respective sequence -- choose such a $\qq$ distinct from $\pp$.

In the case that $n$ is $\qq$-adically close to a TRZ, we use the definition of $\qq$ to bound $\min\{v_\qq(u_n), v_\qq(v_n)\}$. By the same argument as earlier, there are integers $r',m' \geq 0$ such that for all $\ell \in L$, if $y'$ is a TRZ of the subsequence $ \langle u_{Nn+\ell} \rangle_{n=0}^\infty$, then for every $n \in \ZZ$, if $v_\qq(n - y') \geq r'$, then $\min\{v_\qq(u_{Nn+\ell}), v_\qq(v_{Nn+\ell})\} < m'$. Note that for any $y \in Y_u$, we have $\frac{y-\ell}{N}$ is a TRZ of the subsequence $\langle u_{Nn+\ell} \rangle_{n=0}^\infty$ (and analogously for $y \in Y_v$ and $\langle v_{Nn+\ell} \rangle_{n=0}^\infty$). Therefore, we conclude that if $v_\qq(n-y) \geq r'$, then $v_\qq(\frac{n-\ell}{N} - \frac{y-\ell}{N}) \geq r'$ and so $\min\{v_\qq(u_n),v_\qq(v_n)\} < m'$.

Now, in the case where $n$ is $\qq$-adically far from any TRZ $y \in Y$, we use the fact that by \cref{thm:simultaneous}, the only common $\qq$-adic zeros of $u,v$ are TRZs, and hence in $Y$. We claim that there is an integer $\tilde m$ such that for all $n \in \ZZ$, if $v_\qq(n - y) \leq r'$ for all $y \in Y$, we have $\min\{v_\qq(u_n),v_\qq(v_n) \} < \tilde m$. Indeed, suppose no such $\tilde m$ existed, then let $q$ be the rational prime lying below $\qq$ and define the set
\begin{align*}
    B = \bigcap_{y \in Y} \{x \in \ZZ_q : v_\qq(x-y) \leq r' \} \, .
\end{align*}
Since $B$ is a finite intersection of compact sets, $B$ is compact. Now, by assumption there must exist a sequence of integers $n_j \in \ZZ$ such that $n_j \in B$ and $\min\{v_\qq(u_{n_j}),v_\qq(v_{n_j})\} \geq j$ for each $j \in \NN$. By compactness of $B$, there exists a convergent subsequence $\langle n_j\rangle_{j \in J}$, where $J \subseteq \NN$, converging to some $x \in B$. Therefore, we have $v_\qq(n_j - x) \to \infty$ and $\min\{v_\qq(u_{n_j}),v_\qq(v_{n_j})\} \to \infty$, as $j \to \infty$ for $j \in J$. This means $x$ is a $\qq$-adic zero of both $u$ and $v$, yet does not lie in $Y$ by definition of $B$. This contradicts the fact that $Y$ contains all possible common $\qq$-adic zeros of $u$ and $v$, and this proves the claim.

This covers all cases. Indeed, for all $\ell \not\in L$, we have $\min\{v_\pp(u_{Nn+\ell}),v_\pp(v_{Nn+\ell})\} < m $. For all $\ell \in L$, we have $\min\{ v_\qq(u_{Nn+\ell}) , v_\qq(v_{Nn +\ell}) \} < \max\{m', \tilde m\}$. Letting $M = \max\{m,m',\tilde m\}$, we conclude that $(u_n, v_n) \not\equiv (0,0) \mod (\pp \qq)^M$ for all $n \in \ZZ$.
\end{proof}

Note that the theorem does not hold if the ideals $\mathfrak{a}$ are restricted to powers of prime ideals. Indeed, let $u$ be the LRS defined in \cref{ex:two-TRZs}, and $v_n = n(2n-1)$. Recall from \cref{ex:two-TRZs} that for each prime $p>2$, if $N \coloneq \ord_p(4)$ is even, then 0 is a $p$-adic zero of $u$. By \cref{lem:p-adic-zero-def-equal}, this may be witnessed by taking a sequence of integers $n_j \in \NN$ such that $v_p(n_j+\frac{1}{2}) \to \infty$ as $j \to \infty$, then we have $v_p(Nn_j+\frac{N}{2}) \to \infty$ and $v_p(u_{Nn_j+\frac{N}{2}}) \to \infty$ as $j \to \infty$. For such a sequence, we have $v_p(v_{Nn_j+\frac{N}{2}}) \geq v_p(Nn_j + \frac{N}{2}) \geq v_p(n_j + \frac{1}{2}) \to \infty$ as $j \to \infty$, therefore 0 is a $p$-adic zero of both $u$ and $v$. By the same argument, if $\ord_p(4)$ is odd, then $1/2$ is a $p$-adic zero of both $u$ and $v$.

Therefore, $u$ and $v$ have a simultaneous $p$-adic zero for every $p >2$, and so for all integers $k \geq 0$, there is an $n \in \NN$ for which $u_n \equiv v_n \equiv 0 \bmod p^k$.

\section{Closing remarks} \label{sec:conclusion}
\subsection{Effectivity}
For simplicity, we discuss the case of $\QQ$-LRS here. It is plausible that \cref{thm:avoiding_TRZ} and \cref{thm:detecting-TRZ} could be made effective, in the sense that given a TRZ $y \in \QQ$ (that is not an integer zero) of a non-weakly-degenerate $\QQ$-LRS $u$, one could get an effective upper bound $M > 0$ such that there exist primes $p,p' \leq M$ for which $y$ is a $p$-adic zero and is not a $p'$-adic zero.

Indeed, the denominators of TRZs, and the orders of the accompanying roots of unity may be effectively bounded by \cite[Proposition 4.5]{bilu_twisted_2025}. The upper bound $M$ may then be derived from effective versions of \cref{thm:perucca} and \cref{thm:sgobba}. Regarding the latter, \cref{thm:sgobba} is derived from \cite[Theorem 21]{sgobba_divisibility_2023}, which is already effective. Regarding the former, we do not know if an effective version of \cref{thm:perucca} exists. However, the proof is based on Chebotarev's Density Theorem, which may be replaced by an effective version (e.g. \cite{thorner_unified_2019}), so it is plausible an effective version of the theorem could be derived in the context we use it in, though checking that the details would carry through in the correct way is outside the scope of this paper.

Making \cref{thm:avoiding_TRZ} effective would mean that \cref{thm:sim-local-global} could be strengthened to the following statement: let $u,v$ be coprime non-weakly-degenerate $\QQ$-LRS, then subject to the $p$-adic Schanuel conjecture, there is an effective upper bound on an integer $m$ such that there is no $n \in \ZZ$ with $u_n=v_n=0$ if and only if there exists $k \geq 1$, for which $u_n, v_n \not\equiv 0 \bmod m^k$ for all $n \in \ZZ$.

Define the simultaneous Bi-Skolem Problem for coprime LRS to be the following problem: given coprime LRS $u,v$, determine whether there exists $n \in \ZZ$ such that $u_n = v_n = 0$. This is known to be decidable subject to the $p$-adic Schanuel Conjecture \cite{bacik_p-adic_2026}, however, no further complexity upper bound is known.

If one could also derive an effective upper bound on $k$, then this would lead to a non-trivial complexity upper bound for this problem. Indeed, for each $m,k$, the sequences $\langle u_n \bmod m^k \rangle_{n=0}^\infty$ and $\langle v_n \bmod m^k \rangle_{n=0}^\infty$ are periodic; one may easily compute the period and therefore decide whether there is $n \in \NN$ with $u_n \equiv v_n \equiv 0 \bmod m^k$. Then to decide the Simultaneous Bi-Skolem Problem for $u$ and $v$, one simply computes for each $m,k$ below the bound whether there exists $n \in \NN$ with $u_n \equiv v_n \equiv 0 \bmod m^k$; if there is, then $u,v$ must have a simultaneous integer zero $n \in \ZZ$. Otherwise, if there is $m,k$ for which $(u_n,v_n) \not\equiv (0,0) \bmod m^k$, this is a witness for the fact that $u,v$ have no simultaneous zeros. 

However, deriving such an upper bound seems like a difficult problem, and we leave this as an open question.

\subsection{Other Local-Global Principles}
We give a naive heuristic argument for why we should expect the exponential local-global principle (\cref{conj:skolem-conj}) to be true. Let
\begin{align*}
    u_n = \sum_{i=1}^d \alpha_i \lambda_i^n \, ,
\end{align*}
and suppose the period of $u_n \bmod p$ is $N_p$. Let the characteristic polynomial of $u$ be $g \in \ZZ[X]$, then by Chebotarev's Density Theorem, there are infinitely many primes $p$ for which the polynomial $g(X^2)$ splits into linear factors over $\ZZ_p$. For such $p$, the multiplicative order modulo $p$ of each root of $g(X^2)$ divides $p-1$, therefore $\ord_p(\lambda_i) \mid \frac{p-1}{2}$ for each $i=1,\dots, s$, so $N_p \leq \frac{p-1}{2}$. 

For such primes, the reduced sequence $(u_n \bmod p)$ reaches at most $N_p$ elements of $\faktor{\ZZ}{p\ZZ}$. We make two assumptions, which we expect to hold in the absence of ``obstructions'' (we discuss what an obstruction should be later). First, for such primes $p$, we assume that the elements of $\faktor{\ZZ}{p\ZZ}$ reached behave randomly, such that the probability of 0 appearing in the reduced sequence is at most $N_p/p \leq 1/2$. Secondly, we assume that the event of 0 appearing in the reduced sequence is independent for different $p$. This entails that $0$ appearing in the reduced sequence $(u_n \bmod p)$ for $k$ different primes $p$ for which $g(X^2)$ splits in $\ZZ_p$ is at most $(1/2)^k$, so the probability of this occurring for every such prime is 0. Therefore, we expect that there should be some prime $p$ for which $u_n \not\equiv 0 \bmod p$. Note that requiring $u$ to be simple is mandatory for this heuristic argument, as the polynomial term $n$ has period $p$ for every prime $p$, which is too large for the argument to work.

The assumptions made in this heuristic argument are in the absence of ``obstructions''. An obstruction is some reason why these probabilistic assumptions cannot hold. For example, if $u$ has an integer zero $n$, then we certainly have $u_m \equiv 0 \bmod p$ for all $m \equiv n \bmod N_p$. Another obstruction is the existence of twisted rational zeros; we have seen in \cref{ex:two-TRZs} an LRS $u$ with zeros mod $p$ for every prime $p > 2$, due to having two TRZs. However, in the case where an LRS has TRZs but no integer zeros, we expect that one should be able to combine multiple prime ideals as in the proof of \cref{thm:sim-local-global}, hence why we expect Skolem's conjecture to hold. 

We conjecture that TRZs are the \emph{only} possible obstructions to this argument. In particular, we formulate the following conjecture, which we phrase in terms of LRS over number fields for full generality. Let $\KK$ be a number field, let $S$ be a finite set of places of $\KK$, and let
\begin{align*}
    u_n = \sum_{i=1}^d \alpha_i \lambda_i^n
\end{align*}
be a \emph{simple} LRS with $\alpha_1,\dots, \alpha_d, \lambda_1, \dots, \lambda_d, \lambda_1^{-1}, \dots, \lambda_d^{-1} \in \O_S$.
\begin{conjecture}[Exponential Local-Global Principle for TRZs] \label{conj:conj-1}
Suppose that for every non-zero prime ideal $\pp \subseteq \O_S$ there exists $n \in \ZZ$ with $u_n \equiv 0 \bmod \pp$. Then $u$ has a TRZ.
\end{conjecture}
One may also formulate a weaker conjecture, by replacing the prime ideals by all powers of prime ideals.
\begin{conjecture}[Exponential Local-Global Principle for TRZs (Prime Power Version)] \label{conj:conj-2}
Suppose that for every non-zero prime ideal $\pp \subseteq \O_S$ and every integer $k \geq 1$ there exists $n \in \ZZ$ with $u_n \equiv 0 \bmod \pp^k$. Then $u$ has a TRZ.
\end{conjecture}
\cref{conj:conj-2} is equivalent to the statement that if $u$ has a $\pp$-adic zero for every $\pp \subseteq \O_S$, then $u$ has a TRZ. 

We are unable to determine whether this conjecture implies Skolem's conjecture, though we expect that a proof of our conjecture would contain the necessary ideas to prove Skolem's conjecture, perhaps using \cref{thm:avoiding_TRZ} as in the proof of \cref{thm:sim-local-global}. On the other hand, a counterexample would also be very interesting, as we expect it should reveal some further obstruction to the heuristics outlined earlier that we do not currently know about.

\paragraph*{AI Declaration}
ChatGPT 5.6 Sol was used for literature search, and for reviewing the paper for mathematical and writing errors/typos. A non-trivial error was found by ChatGPT in an early version of the proof of \cref{thm:sim-local-global}, which was subsequently corrected by the author. A mild simplification of the last part of the proof of \cref{thm:detecting-TRZ} was also suggested by ChatGPT, and incorporated into the final draft. Everything else, and all text in the article is human-generated.

\bibliography{main.bib}
\end{document}